\documentclass{amsart}
\usepackage{amssymb, amsmath, amsthm, dsfont}
\usepackage{multicol,graphicx}
\usepackage{amsrefs}
\usepackage[arrow, matrix, curve]{xy}
\usepackage{hyperref}
\usepackage{stmaryrd}
\usepackage{extarrows}
\newtheorem{Definition}{Definition}[section]
\newtheorem{Lemma}[Definition]{Lemma}
\newtheorem{Proposition}[Definition]{Proposition}
\newtheorem{Theorem}[Definition]{Theorem}
\newtheorem{Corollary}[Definition]{Corollary}
\newtheorem{Remark}[Definition]{Remark}

\DeclareMathOperator{\cont}{cont}
\DeclareMathOperator{\pro}{pro}
\DeclareMathOperator{\alg}{alg}
\DeclareMathOperator{\sm}{sm}
\DeclareMathOperator{\un}{un}
\DeclareMathOperator{\lfin}{l fin}
\DeclareMathOperator{\ps}{ps}
\DeclareMathOperator{\ann}{ann}

\DeclareMathOperator{\cris}{cr}
\DeclareMathOperator{\Ban}{Ban}

\DeclareMathOperator{\WD}{WD}
\DeclareMathOperator{\Mod}{Mod}
\DeclareMathOperator{\adm}{adm}
\DeclareMathOperator{\ad}{ad}
\DeclareMathOperator{\Supp}{Supp}
\DeclareMathOperator{\projdim}{projdim}
\DeclareMathOperator{\GL}{GL}
\DeclareMathOperator{\SL}{SL}

\DeclareMathOperator{\Ext}{Ext}
\DeclareMathOperator{\Hom}{Hom}
\DeclareMathOperator{\End}{End}
\DeclareMathOperator{\Sym}{Sym}
\DeclareMathOperator{\Alg}{Alg}

\DeclareMathOperator{\tr}{tr}

\DeclareMathOperator{\m-Spec}{m-Spec}
\newcommand{\Qp}{\mathbb{Q}_p}
\newcommand{\Zp}{\mathbb{Z}_p}
\newcommand{\Fp}{\mathbb F_p}

\newcommand{\QQ}{\mathbb Q}

\newcommand{\OO}{\mathcal O}
\newcommand{\mm}{\mathfrak m}
\newcommand{\pp}{\mathfrak p}

\begin{document}

\title[Local proof of Breuil-M{\'e}zard conjecture]{A local proof of the Breuil-M{\'e}zard conjecture in the scalar semi-simplification case}
\author{Fabian Sander}
\abstract{We give a new local proof of the Breuil-M{\'e}zard conjecture in the case of a reducible representation of the absolute Galois group of $\Qp$, $p>2$, that has scalar semi-simplification, via a formalism of Pa{\v{s}}k{\={u}}nas.}
}
\maketitle
\medskip

\section{Introduction}
Let $p>2$ be a prime number, $k$ be a finite field of characteristic $p$ and $L$ a finite extension of $\Qp$ with ring of integers $\OO$ and uniformizer $\varpi$. Let $\rho\colon G_{\Qp}\rightarrow \GL_2(k)$ be a continuous representation of the form
\begin{gather}\label{rho}
\rho(g)=\begin{pmatrix} \chi(g)& \phi(g)\\ 0 &\chi(g)\end{pmatrix}, \forall g\in G_{\Qp},
\end{gather}
so that the semi-simplification of $\rho$ is isomorphic to $\chi\oplus\chi$.
Let $R^{\square}_{\rho}$ denote the associated universal framed deformation ring of $\rho$ and let $\rho^{\square}$ be the universal framed deformation. For any $\pp\in\m-Spec(R^{\square}_{\rho}[1/p])$, the set of maximal ideals, the residue field $\kappa(\pp)$ is a finite extension of $\Qp$. We denote its ring of integers by $\OO_{\pp}$ and get an associated representation $\rho^{\square}_{\pp}\colon G_{\Qp}\rightarrow \GL_2(\OO_{\pp})$ that lifts $\rho$.
Let $\tau\colon I_{\Qp}\rightarrow \GL_2(L)$ be a representation of the inertia group of $\Qp$ with an open kernel, $\psi\colon G_{\Qp}\rightarrow \OO^{\times}$ a continuous character and let $\mathbf{w}=(a,b)$ be a pair of integers with $b>a$. We say that $\rho^{\square}_{\pp}$ is of $p$-adic Hodge type $(\mathbf{w},\tau,\psi)$ if it is potentially semi-stable with Hodge-Tate weights $\mathbf{w}$, $\det\rho_{\pp}\cong \psi\epsilon$, $\left.\psi\right|_{I_{\Qp}}=\epsilon^{a+b}\det\tau$ and $\left.\WD(\rho^{\square}_{\pp})\right|_{I_{\Qp}}\cong\tau$, where $\epsilon$ is the cyclotomic character and $\WD(\rho^{\square}_{\pp})$ is the Weil-Deligne representation associated to $\rho^{\square}_{\pp}$ by Fontaine \cite{MR1293972}.

By a result of Henniart \cite{He1} there exists a unique smooth irreducible ${K:=\GL_2(\Zp)}$-representation $\sigma(\tau)$ and a modification $\sigma^{\cris}(\tau)$ defined by Kisin \cite[1.1.4]{MR2505297} such that for any smooth absolutely irreducible $\GL_2(\Qp)$-representation $\pi$ with associated Weil-Deligne representation $\text{LL}(\pi)$ via the classical local Langlands correspondence, we have $\Hom_K(\sigma(\tau),\pi)\neq0$ (resp. $\Hom_K(\sigma^{\cris}(\tau),\pi)\neq0$) if and only if $\left.\text{LL}(\pi)\right|_{I_{\Qp}}\cong \tau$ (resp. $\left.\text{LL}(\pi)\right|_{I_{\Qp}}\cong \tau$ and the monodromy operator $N$ on $\text{LL}(\pi)$ is trivial). We have $\sigma(\tau)\ncong \sigma^{\cris}(\tau)$ only if $\tau=\chi\oplus\chi$, in which case $\sigma(\tau)=\tilde{\text{st}}\otimes \chi\circ\det$ and $\sigma^{\cris}(\tau)=\chi\circ\det$, where $\tilde{\text{st}}$ is the Steinberg representation of $\GL_2(\Fp)$, inflated to $\GL_2(\Zp)$, and $\chi$ is considered as a character of $\Zp^{\times}$ via local class field theory.
By enlarging $L$ if necessary, we can assume that $\sigma(\tau)\ (\text{resp.}\ \sigma^{\cris}(\tau))$ is defined over $L$. We define $\sigma(\mathbf{w},\tau):=\sigma(\tau)\otimes \Sym^{b-a-1}L^2\otimes \det^a$ and let $\overline{\sigma(\mathbf{w},\tau)}$ be the semi-simplification of the reduction of a $K$-invariant $\OO$-lattice modulo $\varpi$. One can show that $\overline{\sigma(\mathbf{w},\tau)}$ is independent of the choice of the lattice. For every irreducible smooth finite-dimensional $K$-representation $\sigma$ over $k$ we let $m_{\sigma}(\mathbf{w},\tau)$ denote the multiplicity with which $\sigma$ occurs in $\overline{\sigma(\mathbf{w},\tau)}$. Analogously we define $\sigma^{\cris}(\mathbf{w},\tau):=\sigma^{\cris}(\tau)\otimes \Sym^{b-a-1}L^2\otimes \det^a$ and let $m^{\cris}_{\sigma}(\mathbf{w},\tau)$ denote the multiplicity with which $\sigma$ occurs in $\overline{\sigma^{\cris}(\mathbf{w},\tau)}$.

We prove the following theorem.
\begin{Theorem}
Let $p>2$ and let $(\mathbf{w},\tau,\psi)$ be a Hodge type. There exists a reduced $\OO$-torsion free quotient $R_{\rho}^{\square}(\mathbf{w},\tau,\psi)$ (resp. $R_{\rho}^{\square,\cris}(\mathbf{w},\tau,\psi)$) of $R_{\rho}^{\square}$ such that for all $\pp\in\m-Spec(R_{\rho}^{\square}[1/p])$, $\pp$ is an element of $\m-Spec\big(R_{\rho}^{\square}(\mathbf{w},\tau,\psi)[1/p]\big)$ (resp. $\m-Spec\big(R_{\rho}^{\square,\cris}(\mathbf{w},\tau,\psi)[1/p]\big)$) if and only if $\rho^{\square}_{\pp}$ is potentially semi-stable (resp. potentially crystalline) of p-adic Hodge type $(\mathbf{w},\tau,\psi)$. If $R_{\rho}^{\square}(\mathbf{w},\tau,\psi)$ (resp. $R_{\rho}^{\square,\cris}(\mathbf{w},\tau,\psi)$) is non-zero, then it has Krull dimension $5$.

Furthermore, there exists a four-dimensional cycle ${z(\rho)}$ of $R_{\rho}^{\square}$ such that there are equalities of four-dimensional cycles
\begin{equation}\label{eq:semistable1}
z_4\big(R_{\rho}^{\square}(\mathbf{w},\tau,\psi)/(\varpi)\big)=m_{\lambda}(\mathbf{w},\tau)z(\rho),
\end{equation}
\begin{equation}\label{eq:cr31}
z_4\big(R_{\rho}^{\square,\cris}(\mathbf{w},\tau,\psi)/(\varpi)\big)=m^{\cris}_{\lambda}(\mathbf{w},\tau)z(\rho),
\end{equation}
where $\lambda:=\Sym^{p-2}k^2\otimes \chi\circ\det$.
\end{Theorem}

The equality of cycles also implies the analogous equality of Hilbert-Samuel multiplicities. Hence the above theorem proves the Breuil-M{\'e}zard conjecture \cite{MR1944572}, as stated in \cite{MR2505297}, in our case. This case has also been handled by Kisin in \cite{MR2505297} using global methods, see also the errata in \cite{MR3292675}. However, our proof is purely local and the results of this paper, together with works of Pa{\v{s}}k{\=u}nas \cite{MR3306557}, Yongquan Hu and Fucheng Tan \cite{HU1}, the whole conjecture is now proved in the $2$-dimensional case only by local methods, when $p\ge5$.

\section*{Acknowledgement}
This paper was written as a part of my PhD-project. I would like to thank my advisor Prof. Dr. Vytautas Pa{\v{s}}k{\=u}nas for numerous conversations, his patience and his great support. I would also like to thank Yongquan Hu for various helpful comments and suggestions.

\section{Formalism}
We quickly recall a formalism due to Pa{\v{s}}k{\=u}nas used by him to prove the Breuil-M{\'e}zard conjecture for residual representations with scalar endomorphisms in \cite{MR3306557}. Let $R$ be a complete local notherian commutative $\OO$-algebra with residue field $k$. Let $G$ be a $p$-adic analytic group, $K$ be a compact open subgroup and $P$ its pro-$p$ Sylow subgroup. Let $N$ be a finitely generated $R\llbracket K\rrbracket$-module, $V$ be a continuous finite dimensional $L$-representation of $K$, and $\Theta$ be an $\OO$-lattice in $V$ which is invariant under the action of $K$. Let 
\begin{equation}
M(\Theta):=\Hom_{\OO}(\Hom_{\OO\llbracket K\rrbracket}^{\cont}(N,\Hom_{\OO}(\Theta,\OO)),\OO).
\end{equation}
This is a finitely generated $R$-module \cite[Lemma 2.15]{MR3306557}. Let $d$ denote the Krull dimension of $M(\Theta)$. Recall that Pontryagin duality $\lambda\mapsto\lambda^{\vee}$ induces an anti-equivalence of categories between discrete $\OO$-modules and compact $\OO$-modules \cite[(5.2.2)-(5.2.3)]{MR2392026}. For any $\lambda$ in $\Mod^{\sm}_K(\OO)$, the category of smooth $K$-representations on $\OO$-torsion modules, we define 
\begin{equation}
M(\lambda):=\Hom^{\cont}_{\OO\llbracket K\rrbracket}(N,\lambda^{\vee})^{\vee}.
\end{equation}
Then $M(\lambda)$ is also a finitely generated $R$-module \cite[Cor. 2.5]{MR3306557}.
We define $\Mod_G^{\pro}(\OO)$ to be the category of compact $\OO\llbracket K\rrbracket$-modules with an action of $\OO[G]$, such that the restriction to $\OO[K]$ of both actions coincide. Pontryagin duality induces an anti-equivalence of categories between $\Mod^{\sm}_G(\OO)$ and $\Mod_G^{\pro}(\OO)$. For any $R[1/p]$-module $\mathrm{m}$ of finite length, we choose a finitely generated $R$-submodule $\mathrm{m}^0$ with $\mathrm{m}\cong\mathrm{m}^0\otimes_{\OO}L$ and define
 \begin{equation}
 \Pi(\mathrm{m}):=\Hom_{\OO}^{\cont}(\mathrm{m}^0\otimes_RN,L).
 \end{equation} 
 By \cite[Lemma 2.21]{MR3306557}, $\Pi(\mathrm{m})$ is an admissible unitary $L$-Banach space representation of $G$.
\begin{Theorem}[Pa{\v{s}}k{\=u}nas,\cite{MR3306557}]\label{BMC}
Let $\mathfrak{a}$ be the $R$-annihilator of $M(\Theta)$. If the following hold
\begin{enumerate}
\item[(a)] $N$ is projective in $\Mod_K^{\pro}(\OO)$,
\item[(b)] $R/\mathfrak{a}$ is equidimensional and all the associated primes are minimal,
\item[(c)] there exists a dense subset $\Sigma$ of $\Supp M(\Theta)$, contained in $\m-Spec R[1/p]$, such that for all $\mathfrak{n}\in\Sigma$ the following hold:
\begin{enumerate}
\item[(i)] $\dim_{\kappa(\mathfrak{n})}\Hom_K(V,\Pi(\kappa(\mathfrak{n})))=1$,
\item[(ii)] $\dim_{\kappa(\mathfrak{n})}\Hom_K(V,\Pi(R_{\mathfrak{n}}/\mathfrak{n}^2))\leq d$,
\end{enumerate}
\end{enumerate}
then $R/\mathfrak{a}$ is reduced, of dimension $d$ and we have an equality of $(d-1)$-dimensional cycles
\[z_{d-1}(R/(\varpi,\mathfrak{a}))=\sum_{\sigma}m_{\sigma}z_{d-1}(M(\sigma)),\]
where the sum is taken over the set of isomorphism classes of smooth irreducible $k$-representations of $K$ and $m_{\sigma}$ is the multiplicity with which $\sigma$ occurs as a subquotient of $\Theta/\varpi$.
\end{Theorem}

We want to specify the following criterion in our situation, which allows us to check the first two conditions of Theorem \ref{BMC}.

\begin{Theorem}[Pa{\v{s}}k{\=u}nas,\cite{MR3306557}]\label{Criterion}
Suppose that $R$ is Cohen-Macaulay and $N$ is flat over $R$. If
\begin{equation}
\projdim_{\OO\llbracket P\rrbracket}k\hat \otimes_RN+\underset{\sigma}{\max}\{\dim_RM(\sigma)\}\leq \dim R, \label{eq:1}
\end{equation}
where the maximum is taken over all the irreducible smooth $k$-representations of $K$, then the following holds:
\begin{enumerate}
\item[(o)] (\ref{eq:1}) is an equality,
\item[(i)] $N$ is projective in $\Mod_K^{\pro}(\OO)$,
\item[(ii)] $M(\Theta)$ is a Cohen-Macaulay module,
\item[(iii)] $R/\ann_RM(\Theta)$ is equidimensional, and all the associated prime ideals are minimal.
\end{enumerate}
\end{Theorem}

We start with the following setup. Let $\rho\colon G_{\Qp}\rightarrow \GL_2(k)$ be a continuous representation of the form
$\rho(g)=\begin{pmatrix} \chi(g)& \phi(g)\\ 0 &\chi(g)\end{pmatrix}$,
as in \eqref{rho}.
After twisting we may assume that $\chi$ is trivial so that for all $g\in G_{\Qp}$
\begin{gather}
\rho(g)=\begin{pmatrix} 1& \phi(g)\\ 0 &1\end{pmatrix}.
\end{gather}

Let $\psi\colon \Qp^{\times}\rightarrow \OO^{\times}$ be a continuous character with $\psi\epsilon\equiv \mathds{1} \mod \varpi$.
Let $R$ be a complete local noetherian $\OO$-algebra and let 
\begin{equation}
\rho_R\colon G_{\Qp}\rightarrow \GL_2(R)
\end{equation}
be a continuous representation with determinant $\psi\epsilon\colon G_{\Qp}\rightarrow \OO^{\times}$ such that $\rho_R\equiv\rho\mod \mm_R$. Let $R^{\ps,\psi}$ denote the universal deformation ring that parametrizes $2$-dimensional pseudo-characters of $G_{\Qp}$ lifting the trace of the trivial representation and having determinant $\psi\epsilon$. Let $T\colon G_{\Qp}\rightarrow \OO$ be the associated universal pseudo-character. Since $\tr \rho_R$ is a pseudo-character lifting $\tr \rho$, the universal property of $R^{\ps,\psi}$ induces a morphism of $\OO$-algebras
\begin{equation}\label{eq:Rps}
R^{\ps,\psi}\rightarrow R.
\end{equation}



Let from now on $G:=\GL_2(\Qp)$, $P$ the subgroup of upper triangular matrices and $K:=\GL_2(\Zp)$. Let $I_1$ be the subgroup of $K$ which consists of the matrices that are upper unipotent modulo $p$. In particular, $I_1$ is a maximal pro-$p$ Sylow subgroup of $K$. We let $\omega$ be the $\text{mod}\ p$ cyclotomic character, via local class field theory considered as $\omega\colon \Qp^{\times}\rightarrow k^{\times}, x\mapsto x\left|x\right|\mod p$, and define
\begin{equation}
\pi:=(\text{Ind}_P^{G}\mathds{1}\otimes \omega^{-1})_{\sm}.
\end{equation}
We let $\Mod_{G,\psi}^{\sm}(\OO)$ be the full subcategory of $\Mod^{\sm}_G(\OO)$ that consists of smooth $G$-representations with central character $\psi$ and denote by $\Mod_{G,\psi}^{\lfin}(\OO)$ its full subcategory of representations that are locally of finite length. We denote by $\Mod^{\pro}_{G,\psi}(\OO)$ resp. $\mathfrak{C}(\OO)$ the full subcategories of $\Mod_G^{\pro}(\OO)$ that are anti-equivalent to $\Mod_{G,\psi}^{\sm}(\OO)$ resp. $\Mod_{G,\psi}^{\lfin}(\OO)$ via Pontryagin duality.
We see that $\pi$ is an object of $\Mod_{G,\psi}^{\lfin}(\OO)$. Let $\tilde{P}$ be a projective envelope of $\pi^{\vee}$ in $\mathfrak{C}(\OO)$. We define $\tilde{E}:=\End_{\mathfrak{C}(\OO)}(\tilde{P})$. 
Pa{\v{s}}k{\=u}nas has shown in \cite[Cor. 9.24]{MR3150248} that the center of $\tilde{E}$ is isomorphic to $R^{\ps,\psi}$ and 
\begin{equation}
\tilde{E}\cong (R^{\ps,\psi}\hat \otimes_{\OO} \OO\llbracket G_{\Qp}\rrbracket)/J,
\end{equation}
where $J$ is the closure of the ideal generated by $g^2-T(g)g+\psi\epsilon(g)$ for all $g\in G_{\Qp}$ \cite[Cor. 9.27]{MR3150248}. 
The representation $\rho_R$ induces a morphism of $\OO$-algebras $\OO\llbracket G_{\Qp}\rrbracket\rightarrow M_2(R)$. Together with the morphism (\ref{eq:Rps}) we obtain a morphism of $R^{\ps,\psi}$-algebras
\begin{equation}
R^{\ps,\psi}\hat \otimes_{\OO} \OO\llbracket G_{\Qp}\rrbracket\rightarrow M_2(R).
\end{equation}
The Cayley-Hamilton theorem tells us that this morphism is trivial on $J$, so that we get a morphism of $R^{\ps,\psi}$-algebras
\begin{equation}\label{eq:Etilde}
\eta\colon \tilde{E}\rightarrow M_2(R).
\end{equation}
We define 
\begin{equation}\label{Msquare}
M^{\square}(\sigma):=\Hom^{\cont}_{\OO\llbracket K\rrbracket}\big((R\oplus R)\hat\otimes_{\tilde{E},\eta}\tilde{P},\sigma^{\vee}\big)^{\vee}.
\end{equation}
Our goal is to prove the following theorem that enables us to check the condition of Pa{\v{s}}k{\=u}nas' theorem \ref{Criterion} for $N=(R\oplus R)\hat \otimes_{\tilde{E},\eta}\tilde{P}$ in the last section. We let $\projdim_{\OO\llbracket I_1\rrbracket,\psi}$ denote the length of a minimal projective resolution in $\Mod^{\pro}_{I_1,\psi}(\OO)$.
\begin{Theorem}\label{theorem1}
Let $\rho$ and $\rho_R$ be as before. We consider $R$ as an $R^{\ps,\psi}$-module via (\ref{eq:Rps}).
Assume that $\dim R= \dim R^{\ps,\psi}+\dim R/\mm_{R^{\ps,\psi}}R$. Then \[\projdim_{\OO\llbracket I_1\rrbracket,\psi}\big(k\hat \otimes_R((R\oplus R)\hat\otimes_{\tilde{E},\eta}\tilde{P})\big)+\underset{\sigma}{\max}\{\dim_R M^{\square}(\sigma)\}\leq \dim R.\]
In particular, the inequality holds if $R$ is flat over $R^{\ps,\psi}$.
\end{Theorem}
We start with computing the first summand.
\begin{Lemma}\label{lemma1}
\[\projdim_{\OO\llbracket I_1\rrbracket,\psi}\big(k\hat \otimes_R((R\oplus R)\hat \otimes_{\tilde{E},\eta}\tilde{P})\big)=3.\]
\end{Lemma}

\begin{proof}
We have
\[k\hat \otimes_R((R\oplus R)\hat \otimes_{\tilde{E},\eta}\tilde{P})\cong(k\oplus k)\hat \otimes_{\tilde{E}}\tilde{P}.\]
Because of $k\hat\otimes_{\tilde{E}} \tilde{P}\cong \pi^{\vee}$, see \cite[Lemma 9.1]{MR3150248}, and since $\tilde{P}$ is flat over the local ring $\tilde{E}$, $(k\oplus k)\hat \otimes_{\tilde{E}}\tilde{P}$ is an extension of $\pi^{\vee}$ by itself. Thus
\[\projdim_{\OO\llbracket I_1\rrbracket,\psi}\big(k\hat \otimes_R((R\oplus R)\hat \otimes_{\tilde{E},\eta}\tilde{P})\big)=\projdim_{\OO\llbracket I_1\rrbracket,\psi}\pi^{\vee}.\]
The rest of the proof works analogous to the proof of \cite[Prop. 6.21]{MR3306557}, the respective cohomology groups are calculated in \cite[Cor. 10.4]{MR3150248}.
\end{proof}


\begin{Lemma}\label{lemma2}
Let $R$, $N$, $\sigma$ be as before, $\mathrm{m}$ a compact $R$-module. Then \[\Hom_{\OO\llbracket K\rrbracket}^{\cont}(\mathrm{m}\hat\otimes_RN,\sigma^{\vee})^{\vee}\cong\mathrm{m}\hat\otimes_R\Hom_{\OO\llbracket K\rrbracket}^{\cont}(N,\sigma^{\vee})^{\vee}.\]
\end{Lemma}

\begin{proof}
Since $\mathrm{m}$ is compact, we can write it as an inverse limit $\mathrm{m}=\varprojlim \mathrm{m}_i$ of finitely generated $R$-modules.
Also the completed tensor product is defined as an inverse limit, so that we obtain
\begin{align*}
\Hom_{\OO\llbracket K\rrbracket}^{\cont}(\mathrm{m}\hat\otimes_RN,\sigma^{\vee})&\cong \Hom_{\OO\llbracket K\rrbracket}^{\cont}(\varprojlim(\mathrm{m}_i\hat\otimes_RN),\sigma^{\vee})\\
&\cong\Hom_K(\sigma,\varinjlim(\mathrm{m}_i\hat\otimes_RN)^{\vee}).
\end{align*}
The universal property of the inductive limit yields a morphism
\[\varinjlim\Hom_K(\sigma,(\mathrm{m}_i\hat\otimes_RN)^{\vee})\rightarrow \Hom_K(\sigma,\varinjlim(\mathrm{m}_i\hat\otimes_RN)^{\vee}),\]
which is easily seen to be injective. For the surjectivity we have to show that every $K$-morphism from $\sigma$ to $\varinjlim(\mathrm{m}_i\hat\otimes_RN)^{\vee}$ factors through some finite level. But this follows from the fact that $\sigma$ is a finitely generated $K$-representation. This implies
\begin{align*}
\Hom_K(\sigma,\varinjlim(\mathrm{m}_i\hat\otimes_RN)^{\vee})&\cong\varinjlim\Hom_K(\sigma,(\mathrm{m}_i\hat\otimes_RN)^{\vee})\\
&\cong\varinjlim\Hom_{\OO\llbracket K\rrbracket}^{\cont}(\mathrm{m}_i\hat\otimes_RN,\sigma^{\vee}).
\end{align*}
Since the statement holds for finitely generated $\mathrm{m}$ by \cite[Prop. 2.4]{MR3306557}, taking the Pontryagin duals yields
\begin{align*}
\Hom_{\OO\llbracket K\rrbracket}^{\cont}(\mathrm{m}\hat\otimes_RN,\sigma^{\vee})^{\vee}&\cong \varprojlim\Hom_{\OO\llbracket K\rrbracket}^{\cont}(\mathrm{m}_i\hat\otimes_RN,\sigma^{\vee})^{\vee}\\
&\cong\varprojlim \mathrm{m}_i\hat \otimes_R \Hom_{\OO\llbracket K\rrbracket}^{\cont}(N,\sigma^{\vee})^{\vee}\\
&\cong \mathrm{m}\hat \otimes_R \Hom_{\OO\llbracket K\rrbracket}^{\cont}(N,\sigma^{\vee})^{\vee}.
\end{align*}
\end{proof}
For the rest of the section we set $N=\tilde{P}$ so that $M(\sigma)=\Hom^{\cont}_{\OO\llbracket K\rrbracket}(\tilde{P},\sigma^{\vee})^{\vee}$.
\begin{Lemma}\label{lemma3}
Let $\sigma$ be a smooth irreducible $K$-representation over $k$. Then ${M(\sigma)\neq0}$ if and only if $\Hom_K(\sigma,\pi)\neq0$. Moreover,
$\dim_{R^{\ps,\psi}}M(\sigma)\leq1.$
\end{Lemma}

\begin{proof}
By \cite[Cor. 9.25]{MR3150248}, we know that $\tilde{E}$ is a free $R^{\ps,\psi}$-module of rank $4$. Hu-Tan have shown in \cite[Prop. 2.9]{MR3306557} that $M(\sigma)$ is a cyclic $\tilde{E}$-module, thus $M(\sigma)$ is a finitely generated $R^{\ps,\psi}$-module. Furthermore, $M(\sigma)$ is a compact $\tilde{E}$-module, see for example \cite[\S IV.4, Cor.1]{MR0232821}. The same way as in Lemma \ref{lemma2} one can show that 
\begin{equation}\label{Msigma}
\Hom_{\OO\llbracket K\rrbracket}^{\cont}(k\hat\otimes_{\tilde{E}}\tilde{P},\sigma^{\vee})^{\vee}\cong k\hat\otimes_{\tilde{E}}M(\sigma).
\end{equation}
By \cite[Prop. 1.12]{MR3150248}, we have $k\hat\otimes_{\tilde{E}}\tilde{P}\cong \pi^{\vee}$ so that \eqref{Msigma} implies
\begin{equation}
k\hat\otimes_{\tilde{E}}M(\sigma)\cong \Hom_{\OO\llbracket K\rrbracket}^{\cont}(\pi^{\vee},\sigma^{\vee})^{\vee}\cong \Hom_K(\sigma,\pi).
\end{equation}
Hence Nakayama lemma gives us that $M(\sigma)\neq0$ if and only if $\Hom_K(\sigma,\pi)\neq0$.
If this holds, it follows again from \cite[Prop. 2.4]{MR3306557} that, if we let $J$ denote the annihilator of $M(\sigma)$ as $\tilde{E}$-module, there is an isomorphism of rings $\tilde{E}/J\cong k\llbracket S\rrbracket$. Again by \cite[Cor. 9.24]{MR3150248}, $R^{\ps,\psi}$ is isomorphic to the center of $\tilde{E}$. If we let $J^{\ps}$ denote the annihilator of $M(\sigma)$ as $R^{\ps,\psi}$-module, we get an inclusion 
\begin{equation}
R^{\ps,\psi}/J^{\ps}\hookrightarrow \tilde{E}/J\cong k\llbracket S\rrbracket.
\end{equation}
Hence it suffices to show that $\dim_{R^{\ps,\psi}}k\llbracket S\rrbracket\leq 1$, which is equivalent to the existence of an element $x\in \mm_{R^{\ps,\psi}}$ that does not lie in $J^{\ps}$. We assume that $\mm_{R^{\ps,\psi}}\subset J^{\ps}$. Then we have a finite dimensional $k$-vector space $M(\sigma)/\mm_{R^{\ps,\psi}}M(\sigma)\cong M(\sigma)$, on which $\tilde{E}/J\cong k\llbracket S\rrbracket$ acts faithfully, which is impossible.
\end{proof}
The proof of the theorem is now just a combination of the above Lemmas.

\begin{proof}[Proof of Theorem \ref{theorem1}]
Let $\sigma$ be such that $M^{\square}(\sigma)\neq 0$. Then we see from Lemma \ref{lemma2} that \[M^{\square}(\sigma)\cong(R\oplus R)\hat \otimes_{\tilde{E},\eta}M(\sigma).\] 
Since $\tilde{E}$ is a finite $R^{\ps,\psi}$-module by \cite[Cor. 9.17]{MR3150248}, we have
\begin{align*}
\dim_R M^{\square}(\sigma)&=\dim_R(R\oplus R)\hat \otimes_{\tilde{E},\eta}M(\sigma)\\
&\leq \dim_R(R\oplus R)\otimes_{R^{\ps,\psi}}M(\sigma).
\end{align*}
By \cite[A.11]{MR1251956} we know that for a morphism of local rings $A\rightarrow B$ and non-zero finitely generated modules $M,N$ over $A$ resp. $B$, we have
\begin{equation}\label{Bruns}
\dim_BM\otimes_AN\le \dim_AM+\dim_BN/\mm_AN.
\end{equation} 
Since we already know from Lemma \ref{lemma3} that $\dim_{R^{\ps,\psi}}M(\sigma)=1$, we obtain from \eqref{Bruns} that
\[\dim_{R}\left((R\oplus R)\otimes_{R^{\ps,\psi}}M(\sigma)\right)\leq 1+\dim R/\mm_{R^{\ps,\psi}}R.\]
This expression depends only on the structure of $R$ as an $R^{\ps,\psi}$-module and the assumption of the theorem implies
\[\dim_{R}\left((R\oplus R)\otimes_{R^{\ps,\psi}}M(\sigma)\right)\leq 1+\dim R- \dim R^{\ps,\psi}.\]
From the explicit description of $R^{\ps,\psi}$ in \cite[Cor. 9.13]{MR3150248} we know in particular that $R^{\ps,\psi}\cong \OO\llbracket t_1,t_2,t_3\rrbracket$ and thus $\dim R^{\ps,\psi}=4$. The statement is now an immediate consequence of Lemma \ref{lemma1}.
\end{proof}

\section{Flatness}
Let again $\rho\cong \begin{pmatrix} \mathds{1} & \phi \\ 0 & \mathds{1} \end{pmatrix}$. Our goal in this section is to show that the universal framed deformation of $\rho$ with fixed determinant satisfies the conditions of Theorem \ref{theorem1}.
Let $G_{\Qp}(p)$ be the maximal pro-$p$ quotient of $G_{\Qp}$. Since $p>2$, it is a free pro-$p$ group on $2$ generators $\gamma,\delta$ \cite[Thm. 7.5.11]{MR2392026}.
Since the image of $\rho$ is a $p$-group, it factors through $G_{\Qp}(p)$. We have shown in \cite{MR3272032} that the universal framed deformation ring $R^{\square}_{\rho}$ of $\rho$ is isomorphic to $\OO\llbracket x_{11},\hat x_{12},x_{21},t_{\gamma},y_{11},\hat y_{12},y_{21},t_{\delta}\rrbracket$ and the universal framed deformation is given by
\begin{gather}
\rho^{\square}\colon G_{\Qp}(p)\rightarrow \GL_2(R_{\rho}^{\square}),\label{rhosquare}\\
\gamma\mapsto \begin{pmatrix} 1+t_{\gamma}+x_{11}&x_{12}\\x_{21}&1+t_{\gamma}-x_{11}\end{pmatrix},\\
\delta\mapsto \begin{pmatrix} 1+t_{\delta}+y_{11}&y_{12}\\y_{21}&1+t_{\delta}-y_{11}\end{pmatrix},\label{delta}
\end{gather}
where $x_{12}:=\hat x_{12}+[\phi(\gamma)],\ y_{12}:=\hat y_{12}+[\phi(\delta)]$ and $[\phi(\gamma)],[\phi(\delta)]$ denote the Teich\-m\"uller lifts of $\phi(\gamma)$ and $\phi(\delta)$ to $\OO.$ Let $\psi\colon G_{\Qp}\rightarrow \OO^{\times}$ be a continuous character with $\psi\epsilon\equiv 1\ \mbox{mod}\ \varpi$. To find the quotient $R_{\rho}^{\square,\psi}$ of $R_{\rho}^{\square}$ that parametrizes lifts of $\rho$ with determinant $\psi\epsilon$, we have to impose the conditions $\det(\rho^{\square}(\gamma))=\psi\epsilon(\gamma)$ and $\det(\rho^{\square}(\delta))=\psi\epsilon(\delta)$. Therefore, analogous to \cite{MR3272032}, we define the ideal \[I:=\big((1+t_{\gamma})^2-x_{11}^2-x_{12}x_{21}-\psi\epsilon(\gamma),(1+t_{\delta})^2-y_{11}^2-y_{12}y_{21}-\psi\epsilon(\delta)\big)\subset R_{\rho}^{\square,\psi}\] and obtain
\begin{equation}\label{Rdet}
R_{\rho}^{\square,\psi}:=\OO\llbracket x_{11},\hat x_{12},x_{21},t_{\gamma},y_{11},\hat y_{12},y_{21},t_{\delta}\rrbracket /I.
\end{equation}
Let again $R^{\ps,\psi}$ denote the universal deformation ring that parametrizes $2$-dimen\-sional pseudo-characters of $G_{\Qp}$ with determinant $\psi\epsilon$ that lift the trace of the trivial $2$-dimensional representation. Pa{\v{s}}k{\=u}nas has shown in \cite[9.12,9.13]{MR3150248} that $R^{\ps,\psi}$ is isomorphic to $\OO\llbracket t_1,t_2,t_3\rrbracket$ and the universal pseudo-character is uniquely determined by 
\begin{align*}
T\colon G_{\Qp}(p) &\rightarrow \OO\llbracket t_1,t_2,t_3\rrbracket\\
\gamma&\mapsto 2(1+t_1)\\
\delta&\mapsto 2(1+t_2)\\
\gamma\delta&\mapsto 2(1+t_3)\\
\delta\gamma&\mapsto 2(1+t_3).
\end{align*}
Since the trace $T^{\square}$ of $\rho^{\square}$ is a pseudo-deformation of ${2\cdot\mathds{1}}$ to $R^{\square}_{\rho}$, we get an induced morphism
\begin{align}
\phi\colon \OO\llbracket t_1,t_2,t_3\rrbracket\rightarrow &R_{\rho}^{\square,\psi} \label{trace}\\
t_1\mapsto &T^{\square}(\gamma)=t_{\gamma} \label{t1}\\
t_2\mapsto &T^{\square}(\delta)=t_{\delta} \label{t2}\\
t_3\mapsto &T^{\square}(\gamma\delta)=T^{\square}(\delta\gamma)=(1+t_{\gamma})(1+t_{\delta})+\frac{1}{2}z-1,\label{t3}
\end{align}
where $z=x_{12}y_{21}+2x_{11}y_{11}+x_{21}y_{12}$.

\begin{Proposition}\label{Proposition1}
The map $\eqref{trace}$ makes $R_{\rho}^{\square,\psi}$ into a flat $\OO\llbracket t_1,t_2,t_3\rrbracket$-module.
\end{Proposition}

\begin{proof}
Let $\mm$ denote the maximal ideal of $\OO\llbracket t_1,t_2,t_3\rrbracket$. Since $R_{\rho}^{\square,\psi}$ is a regular local ring modulo a regular sequence, it is Cohen-Macaulay. Since $\OO\llbracket t_1,t_2,t_3\rrbracket$ is regular, the statement is equivalent to \[\dim \OO\llbracket t_1,t_2,t_3\rrbracket+\dim R_{\rho}^{\square,\psi}/\mm R_{\rho}^{\square,\psi}=\dim R_{\rho}^{\square,\psi},\] see for example \cite[Thm. 18.16]{MR1322960}. But since ${\dim \OO\llbracket t_1,t_2,t_3\rrbracket=4}$, ${\dim R_{\rho}^{\square,\psi}=7}$ by \eqref{Rdet} and \[R_{\rho}^{\square,\psi}/\mm R_{\rho}^{\square,\psi}\cong k\llbracket x_{11},\hat x_{12},x_{21},y_{11},\hat y_{12},y_{21}\rrbracket/(x_{11}^2+x_{12}x_{21},y_{11}^2+y_{12}y_{21},z)\] by (\ref{trace})-(\ref{t3}), it just remains to prove that \[\dim k\llbracket x_{11},\hat x_{12},x_{21},y_{11},\hat y_{12},y_{21}\rrbracket/(x_{11}^2+x_{12}x_{21},y_{11}^2+y_{12}y_{21},z)=3.\] We distinguish $3$ cases:
If $x_{12}\in (R_{\rho}^{\square,\psi})^{\times}$, we obtain \[R_{\rho}^{\square,\psi}/\mm R_{\rho}^{\square,\psi}\cong k\llbracket x_{11},\hat x_{12},y_{11},\hat y_{12}\rrbracket/(y_{11}^2-y_{12}x_{12}^{-1}(2x_{11}y_{11}-y_{12}x_{11}^2x_{12}^{-1})),\] so that $\{x_{11},\hat x_{12},\hat y_{12}\}$ is a system of parameters for $R_{\rho}^{\square,\psi}/\mm R_{\rho}^{\square,\psi}$. Analogously, if $y_{12}\in (R_{\rho}^{\square,\psi})^{\times}$, then $\{y_{11},\hat y_{12},\hat x_{12}\}$ is a system of parameters.
So the only case left is when $x_{12},y_{12}\notin (R_{\rho}^{\square,\psi})^{\times}$. But it is easy to see that in this case $\{x_{12},y_{21},x_{21}-y_{12}\}$ is a system of parameters for $R_{\rho}^{\square,\psi}/\mm R_{\rho}^{\square,\psi}$, which finishes the proof.
\end{proof}

\section{Locally algebraic vectors}
In this section we want to adapt the strategy of \cite[\S 4]{MR3306557} to show that part c) of Pa{\v{s}}k{\={u}}nas' Theorem \ref{BMC} holds in the following setting.
Let from now on $R:=R_{\rho}^{\square,\psi}$, $\pi\cong (\text{Ind}_{P}^{G}\mathds{1}\otimes \omega^{-1})_{\sm}$, $\tilde{P}$ a projective envelope of $\pi^{\vee}$ in $\mathfrak{C}(\OO)$. Let $N:=(R\oplus R)\hat\otimes_{\tilde{E},\eta}\tilde{P}$, where the $\tilde{E}$-module structure on $R\oplus R$ is induced by $\rho^{\square}$, as in \eqref{eq:Etilde}.

In \cite[\S5.6]{MR3150248} Pa{\v{s}}k{\=u}nas defines a covariant exact functor 
\begin{equation}
\check{\mathbf{V}}\colon \mathfrak{C}(\OO)\rightarrow \Mod_{G_{\Qp}}^{\pro}(\OO),
\end{equation}
which is a modification of Colmez' Montreal functor, see \cite{MR2642409}. It satisfies 
\begin{equation}
\check{\mathbf{V}}((\text{Ind}_{P}^{G}\chi_1\otimes \chi_2\omega^{-1})^{\vee})=\chi_1,
\end{equation}
so that in our case 
\begin{equation}
\check{\mathbf{V}}((\text{Ind}_{P}^{G}\mathds{1}\otimes \omega^{-1})^{\vee})=\mathds{1}.
\end{equation}
For an admissible unitary $L$-Banach space representation $\Pi$ of $G$ with central character $\psi$ and an open bounded $G$-invariant lattice $\Theta$ in $\Pi$, we define 
\begin{equation}
\Theta^d:=\Hom_{\OO}(\Theta,\OO),
\end{equation}
which lies in $\mathfrak{C}(\OO)$. We also define 
\begin{equation}
\check{\mathbf{V}}(\Pi):=\check{\mathbf{V}}(\Theta^d)\otimes_{\OO}L,
\end{equation}
which is independent of the choice of $\Theta$.

\begin{Lemma}\label{Lemma4}
$N$ satisfies the following three properties (see \cite[\S4]{MR3306557}):

\begin{enumerate}
\item[(N0)] $k\hat\otimes_R N$ is of finite length in $\mathfrak{C}(\OO)$ and is finitely generated over $\OO\llbracket K\rrbracket$,
\item[(N1)] $\Hom_{\SL_2(\Qp)}(1,N^{\vee})=0$,
\item[(N2)] $\check{\mathbf{V}}(N)\cong \rho^{\square}$ as $R\llbracket G_{\Qp}\rrbracket$-modules.
\end{enumerate}
\end{Lemma}

\begin{proof}
As we have already seen in the proof of Lemma \ref{lemma1}, $k\hat\otimes_R N$ is an extension of $\pi^{\vee}$ by itself. Since $\pi$ is absolutely irreducible and admissible we get $(N0)$. From \cite[Lemma 5.53]{MR3150248} we obtain that 
\begin{equation}
\check{\mathbf{V}}(\rho^{\square}\hat\otimes_{\tilde{E},\eta}\tilde{P})\cong \rho^{\square}\hat\otimes_{\tilde{E},\eta}\check{\mathbf{V}}(\tilde{P}),
\end{equation}
and since $\check{\mathbf{V}}(\tilde{P})$ is a free $\tilde{E}$-module of rank $1$ by \cite[Cor. 5.55]{MR3150248}, also $(N2)$ holds. For $(N1)$ we notice that $\pi^{\text{SL}_2(\Qp)}=0$. Since $\tilde{P}$ is a projective envelope of $\pi^{\vee}$, $\tilde{P}^{\vee}$ is an injective envelope of $\pi$. Since $G$ acts on $(\tilde{P}^{\vee})^{\text{SL}_2(\Qp)}$ via the determinant, we must have $(\tilde{P}^{\vee})^{\text{SL}_2(\Qp)}=0$.
\end{proof}

\begin{Remark}\label{Rmk2}
Let $\mathrm{m}$ be a $R[1/p]$-module of finite length. Then Lemma \ref{Lemma4} implies that
\[\check{\mathbf{V}}(\Pi(\mathrm{m}))\cong \mathrm{m}\otimes_R\check{\mathbf{V}}(N),\]
see \cite[Rmk. 4.2, Lemma 4.3]{MR3306557}.
\end{Remark}

The following Proposition is analogous to \cite[4.14]{MR3306557} and shows that condition (i) of part c) of Pa{\v{s}}k{\={u}}nas' Theorem \ref{BMC} is satisfied in our setting.
\begin{Proposition}\label{Proposition2}
Let $V$ be either $\sigma(\mathbf{w},\tau)$ or $\sigma^{\cris}(\mathbf{w},\tau)$, let $\pp\in \m-Spec(R[1/p])$ and $\kappa(\pp):=R[1/p]/\pp.$ Then
\[\dim_{\kappa(\pp)}\Hom_K(V,\Pi(\kappa(\pp)))\leq 1.\]
If $V=\sigma(\mathbf{w},\tau)$, then $\dim_{\kappa(\pp)}\Hom_K(V,\Pi(\kappa(\pp)))=1$ if and only if $\rho^{\square}_{\pp}$ is potentially semi-stable of type $(\mathbf{w},\tau,\psi)$.\\
If $V=\sigma^{\cris}(\mathbf{w},\tau)$, then $\dim_{\kappa(\pp)}\Hom_K(V,\Pi(\kappa(\pp)))=1$ if and only if $\rho^{\square}_{\pp}$ is potentially crystalline of type $(\mathbf{w},\tau,\psi)$.
\end{Proposition}

\begin{proof}
Let $F/\kappa(\pp)$ be a finite extension. We have
\[\dim_{\kappa(\pp)}\Hom_K(V,\Pi(\kappa(\pp)))=\dim_F\Hom_K(V\otimes_{\kappa(\pp)} F,\Pi(\kappa(\pp))\otimes_{\kappa(\pp)} F),\]
see for example \cite[Lemma 5.1]{MR3150248}. Thus by replacing $\kappa(\pp)$ by a finite extension, we can assume without loss of generality that $\rho^{\square}_{\pp}$ is either absolutely irreducible or reducible. Since $\rho^{\square}_{\pp}$ is a lift of $\rho\cong \begin{pmatrix} \mathds{1} & \ast \\ 0 & \mathds{1} \end{pmatrix}$ and $N$ satisfies (N0), (N1) and (N2) by Lemma \ref{Lemma4}, the only case that is not handled in \cite[4.14]{MR3306557} is when $\rho^{\square}_{\pp}$ is an extension
\[\xymatrix{0\ar[r]&\chi_1\ar[r]&\rho^{\square}_{\pp}\ar[r]&\chi_2\ar[r]&0},\]
where $\chi_1,\chi_2$ are two characters that have the same Hodge-Tate weight. Such a representation is clearly never of any Hodge-type with distinct Hodge-Tate weights, so it is enough to show that $\dim_{\kappa(\pp)}\Hom_K(V,\Pi(\kappa(\pp))=0$. It follows, for example from \cite[Prop. 3.4.2]{MR2251474}, that $\Pi(\kappa(\pp))$ is an extension of $\Pi_2:=(\text{Ind}_{P}^{G}\chi_2\otimes \chi_1\epsilon^{-1})_{\text{cont}}$ by $\Pi_1:=(\text{Ind}_{P}^{G}\chi_1\otimes \chi_2\epsilon^{-1})_{\text{cont}}$. If we denote the locally algebraic vectors of $\Pi_i$ by $\Pi_i^{\alg}$, then \cite[Prop. 12.5]{MR3150248} tells us that $\Pi_1^{\alg}=\Pi_2^{\alg}=0$. But this implies that also $\Pi(\kappa(\pp))^{\alg}=0$, and since $V$ is a locally algebraic representation, we have
\[\Hom_K(V,\Pi(\kappa(\pp)))\cong \Hom_K(V,\Pi(\kappa(\pp))^{\alg})=0.\]
\end{proof}

To apply Pa{\v{s}}k{\={u}}nas Theorem \ref{BMC}, we have to find a set of 'good' primes of $R[1/p]$ that is dense in $\Supp M(\Theta)$.

\begin{Definition}\label{Definition1}
Let $\Sigma\subset \Supp M(\Theta)\cap \m-Spec(R[1/p])$ consist of all primes $\pp$ such that either  $\Pi(\kappa(\pp))$ is reducible but non-split or $\Pi(\kappa(\pp))$ is absolutely irreducible and $\Pi(\kappa(\pp))^{\alg}$ is irreducible.
\end{Definition}

\begin{Proposition}\label{Proposition3}
$\Sigma$ is dense in $\Supp M(\Theta)$.
\end{Proposition}

\begin{proof}
We already know that $M(\Theta)$ is Cohen-Macaulay by applying Theorem \ref{theorem1} to Pa{\v{s}}k{\={u}}nas' Theorem \ref{Criterion}. Since $R$ is $\OO$-torsion free and $R[1/p]$ is Jacobson, it is enough to show that the dimension of the complement of $\Sigma$ in $\Supp M(\Theta)\cap \m-Spec(R[1/p])$ is strictly smaller than the dimension of $R[1/p]$, which is equal to $4$.

Let first $\pp\in \m-Spec R[1/p]$ be such that $\Pi(\kappa(\pp))$ is absolutely irreducible and $\Pi(\kappa(\pp))^{\alg}$ is reducible. By a result of Colmez \cite[Thm. VI.6.50]{MR2642409} we know that in this case we have $\Pi(\kappa(\pp))^{\alg}\cong \pi\otimes W$, where $W$ is an irreducible algebraic $G$-representation and $\pi\cong (\text{Ind}_P^G\chi\left|.\right|\otimes\chi\left|.\right|^{-1})_{\sm}$ for some smooth character $\chi$. In particular, if the Hodge-Tate weights are $\mathbf{w}=(a,b)$, we have $W\cong \Sym^{b-a-1}L^2\otimes \det^a$. But since $\det\rho^{\square}=\psi\epsilon$, the product of the central characters of $\pi$ and $W$ must be $\psi$, so that we obtain $\chi^2\epsilon^{a+b}=\psi$, which can only be satisfied by a finite number of characters $\chi$. By a result of Berger-Breuil \cite[Cor. 5.3.2]{MR2642406}, the universal unitary completion of $\Pi^{\alg}$ is topologically irreducible in this case and therefore isomorphic to $\Pi$. Hence there are only finitely many absolutely irreducible Banach space representations $\Pi(\kappa(\pp))$ such that $\Pi(\kappa(\pp))^{\alg}$ is reducible. Moreover, all of them give rise to a point $x_{\pp}\in \m-Spec R^{\ps,\psi}[1/p]$ by taking the trace of the associated $G_{\Qp}$-representation $\rho^{\square}_{\pp}=\check{\mathbf{V}}(\Pi(\kappa(\pp)))$. We already know from Proposition \ref{Proposition1} that $R$ is flat over $R^{\ps,\psi}$ and $\dim R/\mm_{R^{\ps,\psi}}R=3$. Thus, above every prime $x_{\pp}$ there lies only an at most $3$-dimensional family of primes $\pp\in\m-Spec R[1/p]$ such that $\Pi(\kappa(\pp))$ is absolutely irreducible and $\Pi(\kappa(\pp))^{\alg}$ is reducible.

Let now $\pp\in \Supp\ M(\Theta)$ be such that, after extending scalars if necessary, $\rho^{\square}_{\pp}$ is split. Hence from Proposition \ref{Proposition2} we know that $\rho^{\square}_{\pp}$ is potentially semi-stable of a Hodge type $(\mathbf{w},\tau,\psi)$ determined by $\Theta$, where $\mathbf{w}=(a,b)$, $\tau=\chi_1\oplus\chi_2$ and $\chi_i\colon I_{\Qp}\rightarrow \GL_2(\overline{\QQ}_p)$ have finite image. We claim that the closed subset of $\m-Spec R_{\rho}^{\square}[1/p]$ consisting of points of the Hodge type above, is of dimension at most $3$. As before, $\rho^{\square}$ factors through the maximal pro-$p$ quotient $G_{\Qp}(p)$ of $G_{\Qp}$, which is a free pro-$p$ group of rank $2$, generated by a 'cyclotomic' generator $\gamma$ and an 'unramified' generator $\delta$. From our assumptions we see that for every representation $\rho^{\square}_\pp$ of the type above there are unramified characters $\mu_1,\mu_2$ such that up to conjugation 
\begin{equation}\label{rhosemi}
\rho^{\square}_{\pp}\sim
\begin{pmatrix}
\epsilon^b\chi_1\mu_1 & 0\\
0 & \epsilon^a\chi_2\mu_2
\end{pmatrix}.
\end{equation}
As in \eqref{rhosquare}, we have $R_{\rho}^{\square}\cong \OO\llbracket x_{11},\hat x_{12},x_{21},t_{\gamma},y_{11},\hat y_{12},y_{21},t_{\delta}\rrbracket$ with the universal framed deformation determined by
\begin{gather}
\rho^{\square}(\gamma)=\begin{pmatrix} 1+t_{\gamma}+x_{11}&x_{12}\\x_{21}&1+t_{\gamma}-x_{11}\end{pmatrix},\label{rhogamma}\\
\rho^{\square}(\delta)=\begin{pmatrix} 1+t_{\delta}+y_{11}&y_{12}\\y_{21}&1+t_{\delta}-y_{11}\end{pmatrix}.\label{rhodelta}
\end{gather}
Since the trace is invariant under conjugation, we get the following identities from \eqref{rhosemi}-\eqref{rhodelta}:
\begin{align}
I_1\colon &\epsilon^b\chi_1(\gamma)+\epsilon^a\chi_2(\gamma)=2(1+t_{\gamma}),\label{eq:I1}\\
I_2\colon &\mu_1(\delta)+\mu_2(\delta)=2(1+t_{\delta}).\label{eq:I2}
\end{align}
We get
\[R_{\rho}^{\square}/(I_1,I_2)\cong \OO\llbracket x_{11},\hat x_{12},x_{21},y_{11},\hat y_{12},y_{21}\rrbracket.\]
Moreover, using \eqref{eq:I1},\eqref{eq:I2}, we get the following relations for the determinants:
\begin{align}
I_3\colon &x_{11}^2+x_{12}x_{21}=\frac{1}{4}(\epsilon^a\chi_1(\gamma)-\epsilon^b\chi_2(\gamma))^2\label{eq:I3},\\
I_4\colon &y_{11}^2+y_{12}y_{21}=\frac{1}{4}(\mu_1(\delta)-\mu_2(\delta))^2.\label{eq:I4}
\end{align}
Since we assume the representation $\rho^{\square}_{\pp}$ to be split, it is, in particular, abelian. This can be summed up in the following relations:
\begin{align}
I_5:\ &0=x_{12}y_{21}-x_{21}y_{12},\\
I_6:\ &0=x_{12}y_{11}-x_{11}y_{12},\\
I_7:\ &0=x_{21}y_{11}-x_{11}y_{21}.
\end{align}
We want to find a system of parameters $\mathcal{S}$ for $R_{\rho}^{\square}/(I_1,\dots,I_7)$ of length at most $4$. 
If $x_{12}\in (R^{\square}_{\rho})^{\times}$, it is easy to check that $\mathcal{S}=\{\varpi,\hat x_{12},\hat y_{12},x_{11}\}$ is such a system. If $y_{12}\in (R^{\square}_{\rho})^{\times}$, we can take $\mathcal{S}=\{\varpi,\hat x_{12},\hat y_{12},y_{11}\}$. In the last case, when $x_{12},y_{12}\in \mm_{R^{\square}_{\rho}}$, which means that $\hat x_{12}=x_{12},\hat y_{12}=y_{12}$, we can take $\mathcal{S}=\{\varpi,x_{12},y_{21},x_{21}-y_{12}\}$. Thus $\dim R_{\rho}^{\square}/(I_1,\dots,I_7)\le 4$ and since $R$ is $\OO$-torsion free, we obtain 
\begin{equation}
\dim R_{\rho}^{\square}[1/p]/(I_1,\dots,I_7)\le 3,
\end{equation}
which proves the claim.
\end{proof}

The next step is to prove that part c)ii) of Pa{\v{s}}k{\=u}nas' Theorem \ref{BMC} is satisfied for all $\pp\in\Sigma$. The following definition is analogous to \cite[4.17]{MR3306557}.

\begin{Definition}\label{Def:RED}
Let $\Ban^{\adm}_{G,\psi}(L)$ be the category of admissible $L$-Banach space representations of $G$ with central character $\psi$ and let $\Pi$ in $\Ban^{\adm}_{G,\psi}(L)$ be absolutely irreducible.
Let $\mathcal{E}$ be the subspace of $\Ext^1_{G,\psi}(\Pi,\Pi)$ that is generated by extensions $0 \rightarrow \Pi\rightarrow E\rightarrow\Pi\rightarrow 0$
such that the resulting sequence of locally algebraic vectors $0 \rightarrow \Pi^{\alg}\rightarrow E^{\alg}\rightarrow\Pi^{\alg}\rightarrow 0$ is exact. 
We say that $\Pi$ satisfies (RED), if $\Pi^{\alg}\neq0$ and $\dim \mathcal{E}\le 1$.
\end{Definition}

The following lemma is a generalization of \cite[Lemma 4.18]{MR3306557} which avoids the assumption $\dim_L\Hom_G(\Pi,E)=1$.

\begin{Lemma}\label{lemma5}
Let $\Pi\in\Ban^{\adm}_{G,\psi}(L)$ be absolutely irreducible. Let $n\ge 1$ and let 
\begin{equation}\label{eq:RED}
0\rightarrow \Pi\rightarrow E\rightarrow \Pi^{\oplus n}\rightarrow 0
\end{equation}
be an exact sequence in $\Ban^{\adm}_{G,\psi}(L)$. Let $V$ be either $\sigma(\mathbf{w},\tau)$ or $\sigma^{\cris}(\mathbf{w},\tau)$. If $\Pi^{\alg}$ is irreducible and $\Pi$ satisfies (RED), then \[\dim_L\Hom_K(V,E)\le\dim_L\Hom_G(\Pi,E)+1.\]
\end{Lemma}

\begin{proof}
Since $\Pi^{\alg}$ is irreducible, we obtain by \cite[Lemma 4.10]{MR3306557} and \cite{He1} that ${\dim_L\Hom_K(V,\Pi)=1}$.
We apply the functors $\Hom_G(\Pi,\_)$ and $\Hom_K(V,\_)$ to the sequence (\ref{eq:RED}) to obtain the following diagram with exact rows.
\[
\xymatrix{
0 \ar[r]& \Hom_G(\Pi,\Pi)\ar[r]\ar[d]^{\cong}&\Hom_G(\Pi,E)\ar[r]\ar@{^{(}->}[d]&\Hom_G(\Pi,\Pi^{\oplus n})\ar[r]\ar[d]^{\cong}&\Ext^1_{G,\psi}(\Pi,\Pi)\ar[d]^{\alpha}\\
0 \ar[r]& \Hom_K(V,\Pi)\ar[r]&\Hom_K(V,E)\ar[r]&\Hom_K(V,\Pi^{\oplus n})\ar[r]&\Ext^1_{K,\psi}(V,\Pi),}
\]
where $\Ext^1$ means the Yoneda extensions in $\Ban^{\adm}_{G,\psi}(L)$ resp. $\Ban^{\adm}_{K,\psi}(L)$.
The diagram yields an exact sequence
\[
\xymatrix{
0\ar[r]&\Hom_G(\Pi,E)\ar[r]&\Hom_K(V,E)\ar[r]&\ker(\alpha),
}\]
and therefore 
\begin{equation}
\dim_L\Hom_K(V,\Pi)\le\dim_L\Hom_G(\Pi,E)+\dim_L\ker(\alpha).
\end{equation}
The irreducibility of $\Pi^{\alg}$ implies that $\ker(\alpha)$ is equal to the space $\mathcal{E}$ of Definition \ref{Def:RED}. Since we assume that $\Pi$ satisfies (RED), we are done.
\end{proof}

\begin{Lemma}\label{lemma6}
Let $\pp\in\Sigma$. If $\End_{G_{\Qp}}(\rho^{\square}_{\pp})=\kappa(\pp)$, then
\[\dim_{\kappa(\pp)}\Hom_{G_{\Qp}}\big(\rho^{\square,\psi}[1/p]/\pp^2,\rho^{\square}_{\pp}\big)=4.\]
If $\rho^{\square}_{\pp}$ is reducible such that there is a non-split exact sequence
\[\xymatrix{0\ar[r]&\delta_2\ar[r]&\rho^{\square}_{\pp}\ar[r]&\delta_1\ar[r]&0,}\]
with $\delta_1\delta_2^{-1}\neq\mathds{1},\epsilon^{\pm 1}$, then
\[\dim_{\kappa(\pp)}\Hom_{G_{\Qp}}\big(\rho^{\square,\psi}[1/p]/\pp^2,\delta_1\big)=4.\]
\end{Lemma}

\begin{proof}
We start with the exact sequence
\begin{equation}\label{eq:basic}
\xymatrix{0\ar[r]&\pp/\pp^2\ar[r]&R[1/p]/\pp^2\ar[r]&\kappa(\pp)\ar[r]&0.}
\end{equation}
Tensoring (\ref{eq:basic}) with $\rho^{\square,\psi}[1/p]$ over $R[1/p]$ and applying the functor $\Hom_{G_{\Qp}}(\_,\rho^{\square}_{\pp})$ yields the exact sequence
\[\Hom_{G_{\Qp}}\big(\rho^{\square,\psi}[1/p]/\pp^2,\rho^{\square}_{\pp}\big)\xlongrightarrow{}\Hom_{G_{\Qp}}\big(\pp/\pp^2\otimes_{R[1/p]}\rho^{\square,\psi}[1/p],\rho^{\square}_{\pp}\big)\xlongrightarrow{\partial} \Ext^1_{G_{\Qp}}(\rho^{\square}_{\pp},\rho^{\square}_{\pp}).\]
Since we assume $\End_{G_{\Qp}}(\rho^{\square}_{\pp})=\kappa(\pp)$, we have 
\[\dim_{\kappa(\pp)}\Hom_{G_{\Qp}}\big(\rho^{\square,\psi}[1/p]/\pp^2,\rho^{\square}_{\pp}\big)=1+\dim_{\kappa(\pp)}\ker(\partial).\]
We see that
\[\ker(\partial)=\{\phi\colon R\rightarrow \kappa(\pp)[\epsilon]\mid \rho^{\square,\psi}[1/p]\otimes_{R[1/p],\phi}\kappa(\pp)[\epsilon]\cong\rho^{\square}_{\pp}\oplus\rho^{\square}_{\pp}\ \text{as}\ G_{\Qp}\text{-reps.}\}.\]
Let $\phi\in\ker(\partial)$ and let $\hat R$ be the $\pp$-adic completion of $R[1/p]$. Then we can identify $\hat R$ with the universal framed deformation ring that parametrizes lifts of $\rho^{\square}_{\pp}$ with determinant $\psi\epsilon$ \cite[(2.3.5)]{MR2600871} and $\phi$ induces a morphism $\hat R\rightarrow \kappa(\pp)[\epsilon]$.
If we denote the adjoint representation of $\rho^{\square}_{\pp}$ by $\ad \rho^{\square}_{\pp}$, there is a natural isomorphism
\begin{gather}
\Hom_{\kappa(\pp)-\Alg}(\hat R,\kappa(\pp)[\epsilon])\cong Z^{1,\psi}(G_{\Qp},\ad \rho^{\square}_{\pp}),
\end{gather}
where $Z^{1,\psi}(G_{\Qp},\ad \rho^{\square}_{\pp})$ denotes the space of cocyles that correspond to deformations with determinant $\psi\epsilon$. Here the morphism  ${\phi\in\Hom_{\kappa(\pp)-\Alg}(\hat R,\kappa(\pp)[\epsilon])}$ that corresponds to a lift $\tilde{\rho}$ of $\rho^{\square}_{\pp}$ is mapped to the cocycle $\Phi$ that appears in the equality
\[\tilde{\rho}(g)=\rho^{\square}_{\pp}(g)(1+\Phi(g)\epsilon).\]
Since $\Ext^1_{G_{\Qp}}(\rho^{\square}_{\pp},\rho^{\square}_{\pp})\cong H^1(G_{\Qp},\ad \rho^{\square}_{\pp})$, we obtain that
\[\ker(\partial)=\{\phi\in Z^{1,\psi}(G_{\Qp},\ad \rho^{\square}_{\pp})\mid \phi=0\ \text{in}\ H^1(G_{\Qp},\ad \rho^{\square}_{\pp})\}.\]
Hence $\ker(\partial)\cong B^{1,\psi}(G_{\Qp},\ad \rho^{\square}_{\pp})$, the corresponding coboundaries.
There is an exact sequence
\begin{equation}\label{seq:ZH}
\xymatrix{
0\ar[r]&(\ad \rho^{\square}_{\pp})^{G_{\Qp}}\ar[r]&\ad \rho^{\square}_{\pp}\ar[r]&Z^1(G_{\Qp},\ad \rho^{\square}_{\pp})\ar[r]&H^1(G_{\Qp},\ad \rho^{\square}_{\pp})\ar[r]&0,
}
\end{equation}
where the middle map is given by $x\mapsto (g\mapsto gx-x)$. 
Since by assumption ${\End_{G_{\Qp}}(\rho_{\pp}^{\square})=\kappa(\pp)}$, we see from (\ref{seq:ZH}) that 
\[\dim_{\kappa(\pp)} B^{1,\psi}(G_{\Qp},\ad \rho^{\square}_{\pp})=3.\]

Let now $\rho^{\square}_{\pp}$ be reducible such that there is a non-split exact sequence
\[\xymatrix{0\ar[r]&\delta_2\ar[r]&\rho^{\square}_{\pp}\ar[r]&\delta_1\ar[r]&0},\]
with $\delta_1\neq\delta_2$. 
Tensoring (\ref{eq:basic}) with $\rho^{\square,\psi}[1/p]$ and applying the functor $\Hom_{G_{\Qp}}(\_,\delta_1)$ gives us an exact sequence
\begin{equation}\label{seq1}
\Hom_{G_{\Qp}}\big(\rho^{\square}[1/p]/\pp^2,\delta_1\big)\xlongrightarrow{}\Hom_{G_{\Qp}}\big(\pp/\pp^2\otimes_{R[1/p]}\rho^{\square,\psi}[1/p],\delta_1\big)\xlongrightarrow{\partial'} \Ext^1_{G_{\Qp}}(\rho^{\square}_{\pp},\delta_1).
\end{equation}
Since $\delta_1\neq\delta_2$ we have $\dim_{\kappa(\pp)}\Hom(\rho^{\square}_{\pp},\delta_1)=1$ and therefore
\begin{equation}\label{dim}
\dim_{\kappa(\pp)} \Hom_{G_{\Qp}}(\rho^{\square,\psi}[1/p]/\pp^2,\delta_1)=1+\dim_{\kappa(\pp)}\ker(\partial').
\end{equation}
Moreover, we obtain isomorphisms 
\begin{align}\label{isos}
\Hom_{G_{\Qp}}\big(\pp/\pp^2\otimes_{R[1/p]}\rho^{\square,\psi}[1/p],\delta_1\big)\cong (\pp/\pp^2)^*&\cong \Hom_{\kappa(\pp)-\Alg}(\hat R^{\square},\kappa(\pp)[\epsilon])\\
&\cong Z^{1,\psi}(G_{\Qp},\ad \rho^{\square}_{\pp}).
\end{align}
From (\ref{seq:ZH}) we obtain again that the kernel of the natural surjection 
\begin{equation}\label{surj}
\xymatrix{Z^1(G_{\Qp},\ad \rho^{\square}_{\pp})\ar@{->>}[r]& H^1(G_{\Qp},\ad \rho^{\square}_{\pp})\cong \Ext^1_{G_{\Qp}}(\rho^{\square}_{\pp},\rho^{\square}_{\pp})}
\end{equation} 
is $3$-dimensional. 
Hence \eqref{seq1}, and \eqref{isos}-\eqref{surj} give us an induced map 
\[\bar{\partial}'\colon \Ext^{1,\psi}_{G_{\Qp}}(\rho^{\square}_{\pp},\rho^{\square}_{\pp})\rightarrow \Ext^1_{G_{\Qp}}(\rho^{\square}_{\pp},\delta_1)\] with
\begin{equation}\label{partials}
\dim_{\kappa(\pp)}\ker(\partial')=3+\dim_{\kappa(\pp)}\ker(\bar{\partial}').
\end{equation}
Since $\End_{G_{\Qp}}(\rho^{\square}_{\pp})=\kappa(\pp)$, also the universal (non-framed) deformation ring $\hat R^{\un}$ of $\rho^{\square}_{\pp}$ exists, that parametrizes deformations of $\rho^{\square}_{\pp}$ with determinant $\psi\epsilon$. 
Therefore we can use the same argument as in the proof of \cite[Lemma 4.20.]{MR3306557}, with $\rho^{\square}_{\pp}$ instead of $\rho_{\pp}^{\un}$, to obtain that $\ker(\bar{\partial}')=\Ext^1_{G_{\Qp}}(\delta_1,\delta_2)/\mathcal{L}$, where $\mathcal{L}$ is the subspace corresponding to $\rho^{\square}_{\pp}$.
Since we assume $\delta_1\delta_2^{-1}\neq \mathds{1},\epsilon^{\pm 1}$, we have $\dim_{\kappa(\pp)}\Ext^1_{G_{\Qp}}(\delta_1,\delta_2)=1$ and obtain from \eqref{dim} and \eqref{partials} that
\[\dim_{\kappa(\pp)}\Hom_{G_{\Qp}}(\rho^{\square,\psi}[1/p]/\pp^2,\delta_1)=4.\]
\end{proof}

\begin{Corollary}\label{Cor1}
Let $V$ be either $\sigma(\mathbf{w},\tau)$ or $\sigma^{\cris}(\mathbf{w},\tau)$ and let $\Theta$ be a $K$-invariant $\OO$-lattice in $V$. Then for all $\pp\in\Sigma$,
\[\dim_{\kappa(\pp)}\Hom_K\big(V,\Pi(R[1/p]/\pp^2)\big)\le 5.\]
\end{Corollary}

\begin{proof}
Let $\pp\in\Sigma$. If $\Pi(\kappa(\pp))$ is absolutely irreducible, then also $\Pi(\kappa(\pp))^{\alg}$ is irreducible. By the same argument as in \cite[Thm. 4.19]{MR3306557} that uses a result of Dospinescu \cite[Thm. 1.4, Prop. 1.3]{Do1}, we obtain that $\Pi(\kappa(\pp))$ satisfies (RED). From the exact sequence 
\begin{equation}\label{eq:start}
\xymatrix{0\ar[r]&\pp/\pp^2\ar[r]&R[1/p]/\pp^2\ar[r]&\kappa(\pp)\ar[r]&0}
\end{equation}
we obtain an exact sequence of unitary Banach space representations
\begin{equation}\label{eq:Banstart}
\xymatrix{0\ar[r]&\Pi(\kappa(\pp))\ar[r]&\Pi(R[1/p]/\pp^2)\ar[r]&\Pi(\kappa(\pp))^{\oplus n}\ar[r]&0}.
\end{equation}
Thus we can apply Lemma \ref{lemma5} and obtain \[\dim_{\kappa(\pp)}\Hom_K\big(V,\Pi(R[1/p]/\pp^2)\big)\le \dim_{\kappa(\pp)}\Hom_G\big(\Pi(\kappa(\pp)),\Pi(R[1/p]/\pp^2)\big)+1.\] The contravariant functor $\check{\mathbf{V}}$ induces an injection
\begin{equation}
\xymatrix{\Hom_G\big(\Pi(\kappa(\pp)),\Pi(R[1/p]/\pp^2)\big) \ar@{^{(}->}[r]& \Hom_{G_{\Qp}}\big(\check{\mathbf{V}}(\Pi(R[1/p]/\pp^2)),\check{\mathbf{V}}(\Pi(\kappa(\pp)))\big).}
\end{equation}
Since the target is isomorphic to $\Hom_{G_{\Qp}}\big(\rho^{\square}[1/p]/\pp^2,\rho^{\square}_{\pp}\big)$ by Remark \ref{Rmk2}, the claim follows from Lemma \ref{lemma6}.

Let now $\Pi(\kappa(\pp))$ be reducible. Then, as in the proof of Proposition \ref{Proposition2}, it comes from an exact sequence 
\begin{equation}\label{eq:gal}
\xymatrix{0\ar[r]&\delta_2\ar[r]&\rho^{\square}_{\pp}\ar[r]&\delta_1\ar[r]&0},
\end{equation}
with $\delta_1\delta_2^{-1}\neq \mathds{1},\epsilon^{\pm 1}$. We obtain an associated exact sequence
\begin{equation}\label{eq:Ban}
\xymatrix{0\ar[r]&\Pi_1\ar[r]&\Pi(\kappa(\pp))\ar[r]&\Pi_2\ar[r]&0},
\end{equation}
where $\check{\mathbf{V}}(\Pi_i)=\delta_i$, $\Pi(\kappa(\pp))^{\alg}=\Pi_1^{\alg}$ and (\ref{eq:Ban}) splits if and only if (\ref{eq:gal}) splits, see \cite[Prop. 3.4.2]{MR2251474}. Furthermore, $\Pi_1$ is irreducible and, again as in \cite[Thm. 4.19]{MR3306557}, $\Pi_1$ satisfies (RED). If we let $E$ be the closure of the locally algebraic vectors in $\Pi(R[1/p]/\pp^2)$, we obtain an isomorphism
\[\Hom_K(V,\Pi(R[1/p]/\pp^2))\cong \Hom_K(V,E).\]
Now (\ref{eq:Banstart}) gives rise to another exact sequences of unitary Banach space representations
\begin{equation}
\xymatrix{0\ar[r]&\Pi_1\ar[r]&E\ar[r]&\Pi_1^{\oplus m}\ar[r]&0.}
\end{equation}
Since $\Pi_1$ satisfies (RED), we can apply Lemma \ref{lemma5} to obtain 
\[\dim_{\kappa(\pp)}\Hom_K(V,E)\le\dim_{\kappa(\pp)}\Hom_G(\Pi_1,E)+1.\]
Because of the inclusions
\[\Hom_G(\Pi_1,E)\hookrightarrow\Hom_G(\Pi_1,\Pi(R[1/p]/\pp^2))\hookrightarrow \Hom_{G_{\Qp}}(\rho^{\square,\psi}[1/p]/\pp^2,\delta_1)\]
we obtain 
\[\dim_{\kappa(\pp)}\Hom_K(V,E)\le\dim_{\kappa(\pp)}\Hom_{G_{\Qp}}(\rho^{\square,\psi}[1/p],\delta_1)+1.\]
But by Lemma \ref{lemma6} $\dim_{\kappa(\pp)}\Hom_{G_{\Qp}}(\rho^{\square,\psi}[1/p],\delta_1)=4$, and we are done.
\end{proof}

Now we are finally able to prove the main theorem. We let again $\chi\colon G_{\Qp}\rightarrow k^{\times}$ be a continuous character and let
\begin{gather*}
\rho\colon G_{\Qp}\rightarrow \GL_2(k)\\
g\mapsto\begin{pmatrix} \chi(g)& \phi(g)\\ 0&\chi(g)\end{pmatrix}.
\end{gather*}

\begin{Theorem}\label{Theorem2}
Let $p>2$ and let $(\mathbf{w},\tau,\psi)$ be a Hodge type. There exists a reduced $\OO$-torsion free quotient $R_{\rho}^{\square}(\mathbf{w},\tau,\psi)$ (resp. $R_{\rho}^{\square,\cris}(\mathbf{w},\tau,\psi)$) of $R_{\rho}^{\square}$ such that for all $\pp\in\m-Spec(R_{\rho}^{\square}[1/p])$, $\pp$ is an element of $\m-Spec\big(R_{\rho}^{\square}(\mathbf{w},\tau,\psi)[1/p]\big)$ (resp. $\m-Spec\big(R_{\rho}^{\square,\cris}(\mathbf{w},\tau,\psi)[1/p]\big)$) if and only if $\rho^{\square}_{\pp}$ is potentially semi-stable (resp. potentially crystalline) of p-adic Hodge type $(\mathbf{w},\tau,\psi)$. If $R_{\rho}^{\square}(\mathbf{w},\tau,\psi)$ (resp. $R_{\rho}^{\square,\cris}(\mathbf{w},\tau,\psi)$) is non-zero, then it has Krull dimension $5$.

Furthermore, there exists a four-dimensional cycle ${z(\rho):=z_4(M(\lambda))}$ of $R_{\rho}^{\square}$, where $\lambda:=\Sym^{p-2}k^2\otimes\chi\circ\det$, such that there are equalities of four-dimensional cycles
\begin{equation}\label{eq:semistable}
z_4\big(R_{\rho}^{\square}(\mathbf{w},\tau,\psi)/(\varpi)\big)=m_{\lambda}(\mathbf{w},\tau)z(\rho),
\end{equation}
\begin{equation}\label{eq:cr3}
z_4\big(R_{\rho}^{\square,\cris}(\mathbf{w},\tau,\psi)/(\varpi)\big)=m^{\cris}_{\lambda}(\mathbf{w},\tau)z(\rho).
\end{equation}
\end{Theorem}

\begin{proof}
We set $N:=(R\oplus R)\hat\otimes_{\tilde{E},\eta}\tilde{P}$, as in Theorem \ref{theorem1}. 
Hence, if we let $\mathfrak{a}$ be the $R$-annihilator of $M(\Theta)$, we obtain from Proposition \ref{Proposition2} and \cite[Prop. 2.22]{MR3306557}, analogous to \cite[Thm. 4.15]{MR3306557}, that for any $K$-invariant $\OO$-lattice $\Theta$ in $\sigma(\mathbf{w},\tau)$ (resp. $\sigma^{\cris}(\mathbf{w},\tau)$) $R/\sqrt{\mathfrak{a}}\cong R_{\rho}^{\square}(\mathbf{w},\tau,\psi)$ (resp. $R/\sqrt{\mathfrak{a}}\cong R_{\rho}^{\square,\cris}(\mathbf{w},\tau,\psi)$). Since $R$ is Cohen-Macaulay, Proposition \ref{Proposition1} shows that we can apply Theorem \ref{theorem1} in our situation. Let $Z$ be the center of $G$ and let $Z_1:=I_1\cap Z$. 
Since $p>2$, there exists a continuous character $\sqrt{\psi}\colon Z_1\rightarrow \OO^{\times}$ with $\sqrt{\psi}^2=\psi$. Twisting by $\sqrt{\psi}\circ \det$ induces an equivalence of categories between $\Mod^{\pro}_{I_1,\psi}(\OO)$ and $\Mod^{\pro}_{I_1/Z_1}(\OO)$. In this way we can use Theorem \ref{theorem1} to show the inequality of Theorem \ref{Criterion} for the setup $G=\GL_2(\Qp)/Z_1$, $K=\GL_2(\Zp)/Z_1$ and $P=I_1/Z_1$. Hence we obtain from Pa{\v{s}}k{\=u}nas' Theorem \ref{Criterion} that the conditions a) and b) of the criterion \ref{BMC} for the Breuil-M{\'e}zard conjecture are satisfied. We let $\Sigma$ be as in Definition \ref{Definition1}. Since we know from Corollary \ref{Proposition3} that $\Sigma$ is dense in $\Supp M(\Theta)$, condition (i) of part c) follows from Proposition \ref{Proposition2}. As already remarked in the proof of Proposition \ref{Proposition3}, we have $\dim M(\Theta)=5$. Thus condition (ii) of part c) is the statement of Corollary \ref{Cor1}. Hence Theorem \ref{BMC} says that there are equations of the form
\begin{equation}\label{lastsemi}
z_4\big(R_{\rho}^{\square}(\mathbf{w},\tau,\psi)/(\varpi)\big)=\sum_{\sigma}m_{\sigma}(\mathbf{w},\tau)z_4(M(\sigma)),
\end{equation}
\begin{equation}\label{lastcris}
z_4\big(R_{\rho}^{\square,\cris}(\mathbf{w},\tau,\psi)/(\varpi)\big)=\sum_{\sigma}m^{\cris}_{\sigma}(\mathbf{w},\tau)z_4(M(\sigma)).
\end{equation}
where the sum runs over all isomorphism classes of smooth irreducible $K$-represen\-ta\-tions over $k$. By Lemma \ref{lemma3} we have that $M(\sigma)\neq 0$ if and only if $\sigma$ lies in the $K$-socle of $\pi$. We let $K_1$ denote the kernel of the projection $K\rightarrow \GL_2(\Fp)$ and let $B$ denote the subgroup of upper triangular matrices of $\GL_2(\Fp)$. Let now $\sigma$ be a smooth irreducible $K$-representation in the $K$-socle of $\pi$. Since $K_1$ is a normal pro-$p$ subgroup of $K$, we must have $\sigma^{K_1}\neq0$ and thus $\sigma=\sigma^{K_1}$. There are isomorphisms of $K$-representations
\begin{equation}
\pi^{K_1}\cong \big((\text{Ind}^K_{P\cap K}\mathds{1}\otimes \omega^{-1})_{\sm}\big)^{K_1}\cong \text{Ind}^{\GL_2(\Fp)}_B\mathds{1}\otimes\omega^{-1},
\end{equation}
and it follows from \cite[Lemma 4.1.3]{MR2128381} that the $K$-socle of $\pi^{K_1}$ is isomorphic to $\Sym^{p-2}k^2\otimes\chi\circ\det$, in particular, it is irreducible. Therefore there is only a single cycle $z(\rho)=z_4(M(\Sym^{p-2}k^2\otimes\chi\circ\det))$ on the right hand side of \eqref{lastsemi} and \eqref{lastcris}.
\end{proof}

\begin{Remark}
If $\tau=\mathds{1}\oplus\mathds{1}$ and $\mathbf{w}=(a,b)$ with $b-a\le p-1$, then the right hand side of \eqref{eq:cr3} is non-trivial if and only if $b-a=p-1$, in which case the Hilbert-Samuel multiplicity of $z(\rho)$ is equal to the multiplicity of $R_{\rho}^{\square,\cris}(\mathbf{w},\mathds{1}\oplus\mathds{1},\psi)/(\varpi)$. In \cite{MR3272032}, we computed that this multiplicity is $1$ if $\rho\otimes\chi^{-1}$ is ramified, $2$ if $\rho\otimes\chi^{-1}$ is unramified and indecomposable, and $4$ if $\rho\otimes\chi^{-1}$ is split.
\end{Remark}

\bibliographystyle{plain}
\bibliography{References}

\end{document}